\documentclass[11pt]{amsart}
\usepackage{amssymb, latexsym, graphicx, mathrsfs, enumerate}

\numberwithin{equation}{section}
\newtheorem{Thm}[equation]{Theorem}
\newtheorem{Prop}[equation]{Proposition}

\newtheorem{Lem}[equation]{Lemma}
\newtheorem{Conj}[equation]{Conjecture}

\theoremstyle{definition}

\newtheorem{Rmk}[equation]{Remark}

\newcommand{\Q}{\mathbb{Q}}

\begin{document}
\title[Extreme residues of Dedekind zeta functions]{Extreme residues of Dedekind zeta functions}
\author[Peter Cho]{Peter J. Cho}
\address{Department of Mathematical Sciences, Ulsan National Institute of Science and Technology, Ulsan, Korea}
\email{petercho@unist.ac.kr}
\author[Henry Kim]{Henry H. Kim$^{\star}$}
\thanks{$^{\star}$ partially supported by an NSERC grant.}
\address{Department of Mathematics, University of Toronto, ON M5S 2E4, CANADA \\
and Korea Institute for Advanced Study, Seoul, Korea}
\email{henrykim@math.toronto.edu}


\begin{abstract} In a family of $S_{d+1}$-fields ($d=2,3,4$), we obtain the true upper and lower bound of the residues of Dedekind zeta functions except for a density zero set. For $S_5$-fields, we need to assume the strong Artin conjecture. We also show that there exists an infinite family of number fields with the upper and lower bound, resp.
\end{abstract}

\maketitle

\section{Introduction}
For a quadratic extension $K=\mathbb{Q}(\sqrt{D})$ with a fundamental discriminant $D$, $Res_{s=1} \zeta_K(s)=L(1,\chi_D)$, where
$\chi_D=\left( \frac{D}{\cdot} \right)$ is the quadratic character. In this case, Littlewood \cite{Lit} obtained the bound
$$
\left(\frac 12+o(1)\right)\frac{\zeta(2)}{e^\gamma \log\log |D|} \leq L(1,\chi_D) \leq (2+o(1))e^\gamma\log\log |D|
$$
under GRH, where $\gamma$ is the Euler-Mascheroni constant. Under the same hypothesis, he also constructed an infinite family of quadratic fields with $L(1,\chi_D) \geq (1+o(1))e^\gamma \log\log |D|$ and an infinite family of quadratic fields with 
$L(1,\chi_D) \leq (1+o(1))\dfrac{\zeta(2)}{e^\gamma \log\log |D|}$. Later, Chowla \cite{Chowla} established the latter omega result unconditionally. It has been conjectured that the true upper and lower bounds are $(1+o(1))e^\gamma \log\log |D|$ and $(1+o(1))\dfrac{\zeta(2)}{e^\gamma \log\log |D|}$, resp. In \cite{MV}, Montgomery and Vaughan considered the distribution of $L(1,\chi_D)$ via random variables which take $\pm 1$ with equal probability. They proposed three conjectures which support the expected bounds. In \cite{GS}, some of the conjectures were proved by Granville and Soundararajan.  

For a number field $K$ of degree $d+1$, the lower bound and the upper bound of $Res_{s=1} \zeta_K(s)$ under GRH and the strong Artin conjecture for $\zeta_K(s)/\zeta(s)$ are
\begin{equation} \label{GRH bound}
 \left(\frac 12+o(1)\right)\frac{\zeta(d+1)}{e^\gamma \log\log |D_K|}\leq Res_{s=1}\zeta_K(s) \leq (2+o(1))^d (e^{\gamma} \log\log|D_K|)^{d},
\end{equation}
where $D_K$ is the discriminant of a number field $K$. The proof of $(\ref{GRH bound})$ is given in Section 3 since at least the upper bound is well-known but it is hard to find its proof in the literature. 

As in the quadratic extension case, we may conjecture that $(1+o(1)) (e^{\gamma}\log\log|D_K|)^{d}$ and $(1+o(1))\dfrac{\zeta(d+1)}{e^\gamma \log\log |D_K|}$ are the true upper and lower bound, resp. In this paper, we show that it is the case except for a density zero set in a family of number fields. A number field $K$ of degree $d+1$ is called a $S_{d+1}$-field if its Galois closure over $\mathbb Q$ is an $S_{d+1}$ Galois extension.  For a $S_{d+1}$-field $K$, we have a decomposition of $\zeta_K(s)$:
$$
\zeta_K(s)=\zeta(s)L(s,\rho,\widehat{K}/Q),
$$
where $\widehat{K}$ is the Galois closure of $K$ over $\Q$ and $\rho$ is the standard representation of $S_{d+1}$. For simplicity, we denote $L(s,\rho,\widehat{K}/Q)$ by $L(s,\rho)$. Hence $Res_{s=1} \zeta_K(s)=L(1,\rho)$. Then, our first main theorem is 

\begin{Thm} \label{true bound} Let $L(X)$ be a set of $S_{d+1}$-fields with $X/2 \leq |D_K|\leq X$, $d+1=3,4$ and $5$. For 
$S_5$-fields, we assume the strong Artin conjecture for $L(s,\rho)$. Then, except for $O(X e^{-c' \frac {\log X}{\log\log X}\log\log\log X})$ $L$-functions for some constant $c'>0$,
\begin{eqnarray*}
 (1+o(1))\frac{\zeta(d+1)}{e^{\gamma} \log\log|D_K| }    \leq L(1,\rho) \leq (1+o(1))(e^{\gamma}\log\log |D_K|)^d. 
\end{eqnarray*}
where $o(1)=O\left(\frac{1}{(\log\log |D_K|)^{1/2}} \right)$.
\end{Thm}

Furthermore, under the same hypothesis, we construct an infinite family of $S_{d+1}$-fields  with extreme residue values.

\begin{Thm}\label{main} Let $d+1=3,4$, and $5$. For $d+1=5$, we assume the strong Artin conjecture. Then 
\begin{enumerate}
\item The number of $S_{d+1}$-fields $K$ of signature $(r_1,r_2)$ with $\frac X2\leq |D_K| \leq X$ for which
\begin{eqnarray*}
L(1,\rho)&=&\prod_{p\leq y} (1-p^{-1})^{-d}\left(1+O\left(\frac 1{(\log\log |D_K|)^{1/2}}\right)\right)\\
&=&(e^{\gamma}\log\log |D_K|)^d \left( 1+O \left(\frac{1}{(\log\log |D_K|)^{1/2}} \right) \right)
\end{eqnarray*}
is $ \geq A(r_2)X\exp\left(-\log |S_{d+1}|\cdot \frac{\log X}{\log\log X}-\log\log\log X\right)$.

\item The number of $S_{d+1}$-fields $K$ of signature $(r_1,r_2)$ with $\frac X2\leq |D_K| \leq X$ for which
\begin{eqnarray*}
L(1,\rho)=\frac{\zeta(d+1)}{e^\gamma \log\log|D_K|}\left(1+ O\left(\frac 1{(\log\log X)^{1/2}}\right)\right)
\end{eqnarray*}
is $\geq A(r_2)X\exp\left(-\log \frac{|S_{d+1}|}{(d+1)}\cdot \frac{\log X}{\log\log X}-\log\log\log X\right)$.
\end{enumerate}
\end{Thm}

We also construct an infinite family of $S_{d+1}$-fields with bounded residues.

\begin{Thm}\label{bounded value} Let $d+1=3,4$, and $5$. For $d+1=5$, we assume the strong Artin conjecture. 

Then the number of $S_{d+1}$-fields $K$ of signature $(r_1,r_2)$ with $\frac X2\leq |D_K| \leq X$ for which
\begin{eqnarray*}
L(1,\rho)= \begin{cases} \zeta(2)^{\frac d2}(1+o(1)), &\text{if $d$ is even}\\ \zeta(2)^{\frac {d-3}2}\zeta(3)(1+o(1)), &\text{if $d\geq 3$ is odd}.\end{cases}.
\end{eqnarray*}
is $\geq A(r_2) X \exp\left(-\log \frac {|S_{d+1}|}{|C|}\cdot \frac {\log X}{\log\log X}-\log\log\log X \right)$,
where 
$$C=\begin{cases} (1,2)(3,4)\cdots (d-1, d), &\text{if $d$ is even}\\ (1,2)(3,4)\cdots (d-4, d-3)(d-2, d-1, d), &\text{if $d$ is odd}\end{cases}.
$$
\end{Thm}

This work is motivated by the work of Lamzouri \cite{Lam, Lam2}, who constructed primitive characters $\chi$ with large values of $L(1,\chi)$. Basically, we follow \cite{Lam, Lam2, GS, MV}. The arguments in \cite{Lam} are easily extended. However, obtaining an analogue of Proposition 2.4 in \cite{Lam} is a main obstacle to extend his method. It is resolved in Proposition \ref{MP}.

\section{Counting number fields with local conditions} \label{Counting}
Let $K$ be a $S_{d+1}$-field of signature $(r_1,r_2)$ for $d+1\geq 3$. We assume that we can count $S_{d+1}$-fields with finitely many local conditions. Namely, let $\mathcal S = (\mathcal{LC}_p)$ be a finite set of local conditions: $\mathcal{LC}_p=\mathcal S_{p,C}$ means that 
$p$ is unramified and the conjugacy class of Frob$_p$ is $C$.  
Define $|\mathcal S_{p,C}|=\frac {|C|}{|S_n|(1+f(p))}$ for some function $f(p)$ which satisfies
$f(p)=O(\frac 1p)$. There are also several splitting types of ramified primes, which are denoted by $r_1,r_2,\dots,r_w$: $\mathcal{LC}_p=\mathcal S_{p,r_j}$ means that $p$ is ramified and its splitting type is $r_j$. 
We assume that there are positive valued functions $c_1(p)$, $c_2(p)$, $\dots$, $c_w(p)$ with $\sum_{i=1}^w c_i(p)=f(p)$ and define 
$|\mathcal S_{p,r_i}|=\frac{c_i(p)}{1+f(p)}$. 
We define the local condition $\mathcal{LC}_p=S_{p,r}$ which means that $p$ is ramified, i.e, $r=r_j$ for some $j$.
Define $|\mathcal S_{p,r}|=\frac{f(p)}{1+f{p}}$. 
Let $|\mathcal S|=\prod_p |\mathcal{LC}_p|$. 

Let $L(X)^{r_2}$ be the set of $S_{d+1}$-fields $K$ of signature $(r_1,r_2)$ with $\frac X2<|D_K| < X$, and
let $L(X;\mathcal S)^{r_2}$ be the set of $S_{d+1}$-fields $K$ of signature $(r_1,r_2)$ with $\frac X2<|D_K| < X$ and the local conditions 
$\mathcal S$. Then we have 
\begin{Conj} \label{counting}
For some positive constants $\delta < 1$ and $\kappa$,
\begin{eqnarray} \label{estimate}
|L(X)^{r_2}| &=& A(r_2) X +O(X^{\delta}),\\
|L(X;\mathcal S)^{r_2}| &=& |\mathcal S| A(r_2) X + O\left(\Big(\prod_{p \in S} p\Big)^{\kappa} X^{\delta} \right), \nonumber
\end{eqnarray}
where the implied constant is uniformly bounded for $p$ and local conditions at $p$. 
\end{Conj}

It is worth noting here that we can control only all the primes up to $c \log X$, where $c<(1-\delta)/\kappa$. If we impose local conditions for all $p\leq c'\log X$ with $c' \geq (1-\delta)/\kappa$, the error term in Conjecture \ref{counting} would be larger than the size of $L(X)^{r_2}$. 

For $S_3$-fields, the conjecture was shown by Taniguchi and Thorne \cite{TT}. In \cite{CK3}\footnote{In \cite{CK3}, we used the Greek letter $\gamma$ in place of $\kappa$. However, $\gamma$ is taken for the Euler-Mascheroni constant in this article.}, we proved that Conjecture \ref{counting} is true for $S_4$ and $S_5$-fields. 

\section{Formula for $L(1,\rho)$ under a certain zero-free region}

In this paper, we assume the strong Artin conjecture, namely, the Artin $L$-function $L(s,\rho)$ is an automorphic representation of 
$GL_d$. This is true for 
$S_3$-fields and $S_4$-fields. It implies the Artin conjecture, namely, $L(s,\rho)$ is entire. For this section, we only need the Artin conjecture. However, in Section 4, we need the strong Artin conjecture in order to use Kowalski-Michel zero density theorem \cite{KM}.
We find an expression of $L(1,\rho)$ as a product over small primes under assumption that $L(s,\rho)$ has a certain zero-free region. Here all the implicit constants only depend on the degree $d$ of $L(s,\rho)$. 

For $Re(s)>1$, $L(s,\rho)$ has the Euler product:
\begin{equation*}
L(s,\rho)=\prod_{p} \prod_{i=1}^d \left( 1-\frac{\alpha_i(p)}{p^s}\right)^{-1}.
\end{equation*} 

Then, for $Re(s)>1$, 
$$\log L(s,\rho)=\sum_{n=2}^\infty \frac {\Lambda(n)a_\rho(n)}{n^s \log n},
$$
where $a_\rho(p^k)=\alpha_1(p)^k+\cdots+\alpha_d(p)^k$. First, we show that when $L(s,\rho)$ has a certain zero-free region, 
the value $\log L(1,\rho)$ is determined by a short sum. 

\begin{Prop} \label{App1}
If $L(s,\rho)$ is entire and is zero-free in the rectangle $[\alpha, 1] \times [-x,x]$, where $x=(\log N)^\beta$,\, $\beta(1-\alpha)>2$,
 and $N$ is the conductor of $\rho$, then 
\begin{equation}\label{log-D}
\log L(1,\rho)=\sum_{n< x} \frac {\Lambda(n)a_\rho(n)}{n\log n} +O((\log N)^{-1}).
\end{equation}
\end{Prop}

\begin{proof}
By Perron's formula,
$$\frac 1{2\pi i}\int_{c-ix}^{c+ix} \log L(1+s,\rho)\, \frac {x^s}s\, ds=\sum_{n< x} \frac {\Lambda(n)a_\rho(n)}{n\log n}+O\left(\frac {\log x}x\right).
$$
where $c=\frac 1{\log x}$. 

Now move the contour to $Re(s)=\alpha-1+\frac 1{\log x}$. We get the residue $\log L(1,\rho)$ at $s=0$. So the left hand side is $\log L(1,\rho)$ plus
$$\frac 1{2\pi i}\left(\int_{c-ix}^{\alpha-1+c-ix}+\int_{\alpha-1+c-ix}^{\alpha-1+c+ix}+\int_{\alpha-1+c+ix}^{c+ix}\right) \log L(1+s,\rho)\frac {x^s}s\, ds.
$$

In order to estimate $|\log L(s,\rho)|$ for $\alpha+c\leq Re(s)\leq 1+c$, we follow \cite[Lemma 8.1]{GS1}:
Consider the circles with centre $2+it$ and radii $r=2-\sigma<R=2-\alpha$. By the assumption, $\log L(s,\rho)$ is holomorphic inside the larger circle. By Daileda \cite[page 222]{D}, for $\frac 12<Re(s)\leq \frac 32$, $|L(s,\rho)|\leq N^{\frac 12}(|s|+1)^{\frac d2}$. Hence $Re \log L(s,\rho)=\log |\log L(s,\rho)|\ll \log N+\log (|s|+1).$ Clearly, if $Re(s)\geq \frac 32$, $|\log L(s,\rho)|=O(1)$. 
By the Borel-Carath\'eodory theorem, 
$$|\log L(s,\rho)|\leq \frac {2r}{R-r} \max_{|z-(2+it)|=R} Re \log L(z,\rho)+\frac {R+r}{R-r} |\log L(2+it,\rho)|\ll (\log x) (\log N+\log (|s|+1)).
$$
Hence the integral is majorized by $x^{\alpha-1} (\log N)(\log x)^2$. Since $\beta(1-\alpha)>2$, $x^{\alpha-1} (\log N)(\log x)^2\ll (\log N)^{-1}$.
\end{proof}

\begin{Rmk}
Assume that $L(s,\rho)$ satisfies GRH. Take $\alpha=1/2+\epsilon^2$ and $\beta=2+\epsilon$. Then, from the above proof, we can see that
$$
\log L(1,\rho)=\sum_{n< (\log N)^{2+\epsilon}} \frac {\Lambda(n)a_\rho(n)}{n\log n} + 
O\left( \frac{\log\log N}{(\log N)^{\frac {\epsilon}{2}-(2\epsilon^2+\epsilon^3)}}\right),
$$
for any $\epsilon>0$. 
\end{Rmk}

Now, using Proposition \ref{App1}, we express $L(1,\rho)$ as a product over small primes. We omit $p$ from $\alpha_i(p)$ for simplicity. 
\begin{equation}\label{sum}
\sum_{n<x} \frac {\Lambda(n)a_\rho(n)}{n\log n}=\sum_{k, p^k<x} \frac {\alpha_1^k+\cdots+\alpha_d^k}{k p^k}
=\sum_{p<x}\sum_{i=1}^d \sum_{k<\frac {\log x}{\log p}} \frac 1k (\alpha_i p^{-1})^k.
\end{equation}
Here
$$
\sum_{k<\frac {\log x}{\log p}} \frac 1k (\alpha_i p^{-1})^k=-\log(1-\alpha_i p^{-1})+A_p,
$$
where
$$|A_p|\leq \sum_{k\geq \frac {\log x}{\log p}} \frac 1k p^{-k}\leq \frac {\log p}{\log x}\cdot \frac {p^{-\frac {\log x}{\log p}}}{1-p^{-1}}.
$$
Here $p^{\frac {\log x}{\log p}}=x$.
Hence
$$(\ref{sum})=-\sum_{p<x} \sum_{i=1}^d \log (1-\alpha_i p^{-1})+d\sum_{p<x} A_p.
$$
Here
$$\sum_{p<x} |A_p|\leq \frac 1{x\log x} \sum_{p<x} \frac {\log p}{1-p^{-1}}\leq \frac 2{\log x}.
$$
Therefore, it is summarized as follows:
\begin{Prop} \label{App2}
If $L(s,\rho)$ is entire and is zero-free in the rectangle $[\alpha, 1] \times [-x,x]$, where $x=(\log N)^\beta$,\, $\beta(1-\alpha)>2$,
 and $N$ is the conductor of $\rho$, then 
\begin{equation}\label{L(1)}
L(1,\rho)=\prod_{p<x} \prod_{i=1}^d (1-\alpha_i p^{-1})^{-1} \left(1+O\left(\frac 1{\log x}\right)\right).
\end{equation}
Furthermore, if $L(s,\rho)$ satisfies GRH, then
$$
L(1,\rho)=\prod_{p<(\log N)^{2+\epsilon}} \prod_{i=1}^d (1-\alpha_i p^{-1})^{-1} \left(1+O\left(\frac 1{\log\log N}\right)\right).
$$
\end{Prop}

In order to find the upper and lower bound of $L(1,\rho)$, we examine the Euler product: Let $C$ be a conjugacy class of $S_{d+1}$, and let $C$ be a product of $d_1,\cdots, d_k$ cycles, where $d_i\geq 1$ for all $i$ and $d_1+\cdots+d_k=d+1$. Then if $Frob_p\in C$, 
$(1-X)\prod_{i=1}^d (1-\alpha_i X)=(1-X^{d_1})\cdots (1-X^{d_k})$. 
Hence
$$\prod_{i=1}^d (1-\alpha_i p^{-1})^{-1}=(1-p^{-1}) (1-p^{-d_1})^{-1}\cdots (1-p^{-d_k})^{-1}.
$$
Now we use Mertens' theorem:
$$\prod_{p\leq y} (1-p^{-1})^{-1}=e^{\gamma} (1+o(1))\log y.
$$
Also $\prod_{p\leq y} (1-p^{-n})^{-1}=\zeta(n)(1+O(\frac 1{y\log y}))$ if $n\geq 2$.

Hence the upper bound of $\prod_{i=1}^d (1-\alpha_i p^{-1})^{-1}$ is when $C=1$, and it is $(1-p^{-1})^{-d}$. 
The lower bound is when $C=(1,\cdots, d+1)$, and it is $(1-p^{-1})(1-p^{-d-1})^{-1}$.
Moreover, it takes only the values $(1-p^{-e_1})^{-a_1}\cdots (1-p^{-e_l})^{-a_l}(1-p^{-1})^{a_0}$, where $e_1,...,e_l\geq 2$, and
$-d\leq a_0\leq 1$. Here $a_0=1$ only when $a_1e_1+\cdots+a_le_l=d+1$. We summarize it as

\begin{equation}\label{euler}
(1-p^{-1})(1-p^{-d-1})^{-1}\leq \prod_{i=1}^d (1-\alpha_i p^{-1})^{-1}\leq (1-p^{-1})^{-d}.
\end{equation}

We note that (\ref{euler}) is true even if $p$ is ramified, i.e., when some of $\alpha_i$'s are zero. Hence by the above proposition, under GRH and the strong Artin conjecture for $L(s,\rho)$, for any $\epsilon >0$,

\begin{eqnarray*}
\frac{\zeta(d+1)}{(2+\epsilon)e^\gamma \log\log N}\left(1+o(1)\right) \leq L(1,\rho)\leq \left( e^\gamma (2+\epsilon)\log\log N \right)^d \left(1+o(1)\right).
\end{eqnarray*}

Since $\epsilon$ is arbitrarily small, we showed
\begin{eqnarray*}
 \left(\frac 12+o(1)\right)\frac{\zeta(d+1)}{e^\gamma \log\log N} \leq L(1,\rho) \leq (2+o(1))^d (e^\gamma \log\log N)^d.
\end{eqnarray*}

\section{Extreme residue values}

\subsection{True upper and lower bound}

For simplicity, we denote $L(X)^{r_2}$ by $L(X)$.
Let $y=c_1\log X$ with $c_1>0$. 
Recall that in Proposition \ref{App1}, the conductor of $L(s,\rho)$ is $|D_K|$, and $\frac X2<|D_K|<X$, and $x=(\log X)^{\beta}$ for some $\beta$.

In this section we show that except for $O(X e^{-c' \frac {\log X}{\log\log X}\log\log\log X})$ in $L(X)$,
the lower bound and upper bound on $L(1,\rho)$ are 
$$(1+o(1))\frac{\zeta(d+1)}{e^{\gamma } (\log\log|D_K|)} ,\quad (1+o(1)) (e^{\gamma} \log\log |D_K|)^d,\quad \text{resp.}
$$

We apply Kowalski-Michel zero density theorem \cite{KM} to the family $L(X)$. Then except for $O\left((\log X)^{\beta B}X^{(\frac{5d}{2}+1) \frac{1-\alpha}{2\alpha-1}} \right)$ $L$-functions, every $L$-function $L(s,\rho)$ in $L(X)$ is zero-free on $[\alpha,1] \times [-(\log X)^\beta, (\log X)^\beta]$ with $\beta(1-\alpha)>2$. Here $B$ is a constant depending on the family $L(X)$. We refer to \cite{CK-JNT} for the detail. 

Since except for $O\left((\log X)^{\beta B}X^{(\frac{5d}{2}+1) \frac{1-\alpha}{2\alpha-1}} \right)$ $L$-functions, 
the $L$-functions in $L(X)$ have the desired zero-free region, we apply Proposition \ref{App2} to the $L$-functions in $L(X)$ to obtain
\begin{equation*}
L(1,\rho)=\prod_{p<x} \prod_{i=1}^d (1-\alpha_i p^{-1})^{-1} \left(1+O\left(\frac 1{\log x}\right)\right).
\end{equation*}
Since
$$\sum_{y<p<x} \frac 1{p^2}\leq \sum_{p>y} \frac 1{p^2}\leq \frac 2{y\log y},
$$
we can show
$$
\prod_{y<p<x} \prod_{i=1}^d (1-\alpha_i p^{-1})^{-1}= \exp\left(\sum_{y<p<x} \frac {a_{\rho}(p)}p\right) \left(1+ O\left(\frac 1{y\log y}\right)\right).
$$

We prove
\begin{Prop}\label{Lam}
Except for $O(X e^{-c' \frac {\log X}{\log\log X}\log\log\log X})$ L-functions in $L(X)$ for some constant $c'>0$, $L$-functions in $L(X)$ satisfy
\begin{equation}\label{small-sum}
\left|\sum_{y<p<x} \frac {a_{\rho}(p)}p\right| \leq \frac 1{(\log\log X)^{1/2}}.
\end{equation}
\end{Prop}

Hence, for $L$-functions which have the desired zero-free region and satisfy $(\ref{small-sum})$, 
\begin{eqnarray*}
L(1,\rho)=\prod_{p\leq y} \prod_{i=1}^d \left( 1-\alpha_i p^{-1}\right)^{-1}\left(1+ \frac{1}{(\log\log |D_K|)^{1/2} } \right).
\end{eqnarray*}

This and (\ref{euler}) implies immediately Theorem \ref{true bound}. 

In order to prove Proposition \ref{Lam}, we follow the idea in \cite{Lam}. Namely we prove
\begin{Prop}\label{MP} Let $y=c_1 \log X$ and $r \leq c_2 \frac{\log X}{\log\log X}$ for some positive constants $c_1$ and $c_2$. Then, 
$$\sum_{\rho\in L(X)} \left(\sum_{y<p<x} \frac {a_{\rho}(p)}p\right)^{2r}\ll 2^{2r-1} d^{2r}  \frac {(2r)!}{r!} \frac{2^{2r}}{(y \log y)^r}X,
$$
with an absolute implied constant.
\end{Prop}
By Stirling's formula,
$\displaystyle 2^{2r-1} d^{2r} \frac {(2r)!}{r!} \frac{2^{2r}}{(y \log y)^r} \ll \left( \frac{cd^2r}{ y \log y}\right)^r$ for a constant $c$.

\begin{proof}
By multinomial formula, the left hand side is
\begin{equation}\label{multi}
\sum_{\rho\in L(X)} \sum_{u=1}^{2r}  \frac 1{u!} {\sum}_{r_1,...,r_u}^{(1)}\frac {(2r)!}{r_1!\cdots r_u!}  {\sum}_{p_1,...,p_u}^{(2)} \frac {a_{\rho}(p_1)^{r_1}\cdots a_{\rho}(p_u)^{r_u}}{p_1^{r_1}\cdots p_u^{r_u}},
\end{equation}
where $\sum_{r_1,...,r_u}^{(1)}$ means the sum over the $u$-tuples $(r_1,...,r_u)$ of positive integers such that $r_1+\cdots+r_u=2r$, and
$\sum_{p_1,...,p_u}^{(2)}$ means the sum over the $u$-tuples $(p_1,...,p_u)$ of distinct primes such that $y<p_i<x$ for each $i$.
Write
$$(\ref{multi})=\sum_{u=1}^{2r} {\sum}_{r_1,...,r_u}^{(1)}\frac {(2r)!}{r_1!\cdots r_u!} \frac 1{u!}  {\sum}_{p_1,...,p_u}^{(2)}
\frac 1{p_1^{r_1}\cdots p_u^{r_u}} \left(\sum_{\rho\in L(X)} a_{\rho}(p_1)^{r_1}\cdots a_{\rho}(p_u)^{r_u}\right).
$$

We will show that for any composition $r_1+r_2+\cdots+r_u=2r$,
\begin{equation} \label{composition-bound}
\frac {(2r)!}{r_1!\cdots r_u!} \frac 1{u!}  {\sum}_{p_1,...,p_u}^{(2)}
\frac 1{p_1^{r_1}\cdots p_u^{r_u}} \left(\sum_{\rho\in L(X)} a_{\rho}(p_1)^{r_1}\cdots a_{\rho}(p_u)^{r_u}\right)
\ll d^{2r}X \frac{(2r)!}{r!}\frac{2^{2r}}{(y \log y)^r}.
\end{equation}

Since the number of compositions of $2r$ is $2^{2r-1}$, it implies that
$$
(\ref{multi}) \ll 2^{2r-1} d^{2r}  \frac {(2r)!}{r!} \frac {2^{2r}}{(y\log y)^r} X.
$$

First, we consider compositions with $r_i \geq 2$ for all $i$. Then by using the trivial bound,
\begin{eqnarray*}
&& {\sum}_{p_1,...,p_u}^{(2)}
\frac 1{p_1^{r_1}\cdots p_u^{r_u}} \left(\sum_{\rho\in L(X)} a_{\rho}(p_1)^{r_1}\cdots a_{\rho}(p_u)^{r_u}\right) 
\ll d^{2r}X \left(\sum_{y<p_1<x} \frac 1{p_1^{r_1}}\right)\cdots \left(\sum_{y<p_u<x} \frac 1{p_u^{r_u}}\right) \\
&&\ll d^{2r}X \frac{2^{2r}}{(y \log y)^r} \left( \frac{\log y}{y}\right)^{r-u}. 
\end{eqnarray*}

Hence $(\ref{composition-bound})$ is proved once we show that for any $r_1,...,r_u$ such that $r_1+\cdots+r_u=2r$, and $r_i\geq 2$ for all $i$,
\begin{equation*}
\frac{1}{u!r_1! \cdots r_u!} \left( \frac{ \log y}{y} \right)^{r-u} \leq \frac{1}{r!},
\end{equation*}
or equivalently
\begin{equation} \label{composition-inequality}
\frac{r!}{u!r_1! \cdots r_u!} \leq \left( \frac{y}{\log y} \right)^{r-u}.
\end{equation}

Since $r_i \geq 2$ for all $i=1,2,\dots,u$, we have $u \leq r$. Since $y=c_1 \log X$ and $r\leq c_2 \frac{\log X}{\log \log X}$, $r\leq \frac{y}{\log y}$ for sufficiently small $c_2$. 
Then 
$$\frac{r!}{u!r_1! \cdots r_u!}\leq \frac {r!}{u!}=r(r-1)\cdots (r-u+1) \leq r^{r-u}\leq \left( \frac{y}{\log y} \right)^{r-u}.
$$

Next, suppose $r_i=1$ for some $i$. We may assume that $r_1+\cdots+r_m+r_{m+1}+\cdots +r_u=2r$, $r_1=...=r_{m}=1$, and $r_{m+1}>1,...,r_u>1$.
First, we need a technical combinatorial lemma.

\begin{Lem} \label{tech lemma}
Let $r_i$'s be as above. Then
\begin{eqnarray}\label{composition-ineq2}
 \frac{1}{u!}\cdot \frac{1}{r_1!r_2!\dots r_m! r_{m+1}! \dots r_u!}\cdot \frac{y^u}{y^{m+r}}\cdot \frac{(\log y)^r}{(\log y)^u} \leq \frac{1}{r!}.
\end{eqnarray}
\end{Lem}
\begin{proof}
First, we assume that $m$ is even. 
Then since $r_{m+1},\dots,r_u \geq 2$, and $r_{m+1}+\cdots+r_u=2r-m$, by (\ref{composition-inequality}),
\begin{eqnarray*}
\frac{\left( \frac{2r-m}{2}\right)!}{(u-m)!r_{m+1}!\dots r_u!} \leq
\left( \frac{y}{\log y} \right)^{(r-m/2)-(u-m)}\leq \left( \frac{y}{\log y}  \right)^{r+m/2-u} 
\end{eqnarray*}
Hence
$$
\frac{1}{r_{m+1}!\dots r_u!} \leq \frac{(u-m)!}{(r-m/2)!}\left( \frac{y}{\log y}  \right)^{r+m/2-u}. 
$$
So
\begin{eqnarray*}
\frac{1}{u!}\cdot \frac{1}{r_1!r_2!\dots r_m! r_{m+1}! \dots r_u!} \frac{y^u}{y^{m+r}}\frac{(\log y)^r}{(\log y)^u} &\leq& \frac{(u-m)!}{u!} \frac{1}{(r-m/2)!}\left( \frac{y}{\log y}  \right)^{r+m/2-u}\frac{y^u}{y^{m+r}}\frac{(\log y)^r}{(\log y)^u} \\
&\leq & \frac{(u-m)!}{u!} \frac{1}{(r-m/2)!} \frac{1}{(y \log y)^{m/2}}
\end{eqnarray*}

Since $r < y$ and $\frac{(u-m)!}{u!} <1$, 
$$
\frac{r!}{(r-\frac{m}{2})!} \frac{(u-m)!}{u!} \leq (y \log y)^{m/2}.
$$
This implies
$$
 \frac{(u-m)!}{u!} \frac{1}{(r-m/2)!} \frac{1}{(y \log y)^{m/2}} \leq \frac{1}{r!}.
$$
Hence we have $(\ref{composition-ineq2})$. 

When $m$ is odd, we consider a composition of $2r-m+3$ of the form:
$$
r'_{m+1}=r_{m+1}, r'_{m+2}=r_{m+2}, \dots, r'_u=r_u, \mbox{ and } r'_{u+1}=3.  
$$

With this composition, by (\ref{composition-inequality}),
\begin{eqnarray*}
\frac{\left( \frac{2r-m+3}{2}\right)!}{(u-m+1)!r_{m+1}!\dots r_{u}!3! }=\frac{\left( \frac{2r-m+3}{2}\right)!}{(u-m+1)!r'_{m+1}!\dots r'_u!r_{u+1}'! } \leq
\left( \frac{y}{\log y} \right)^{r+m/2+1/2-u}.
\end{eqnarray*}

As we did for the case of even $m$, since $ r < y$ and $\frac{(u-m+1)!}{u!} \leq 1$,
we have
$$
\frac{r!}{(r-\frac{m-3}{2})!} \frac{(u-m+1)!}{u!} \leq \frac 16(y \log y)^{\frac {m-1}2}\log y.
$$
This implies $(\ref{composition-ineq2})$. 
\end{proof}

Recall that we are treating a composition $r_1+r_2+\cdots +r_u=2r$ with $r_1=r_2=\cdots=r_m=1$. Let $N$ be the number of conjugacy classes of $G$, and partition the sum $\sum_{\rho\in L(X)}$ into $(N+w)^{u}$ sums, namely, given
$(\mathcal S_1,...,\mathcal S_{u})$, where $\mathcal S_i$ is either $\mathcal S_{p_i,C}$ or $\mathcal S_{p_i,r_j}$,
we consider the set of $\rho\in L(X)$ with the local conditions $\mathcal S_i$ for each $i$. Note that in each such partition, $a_{\rho}(p_1)^{r_1}\cdots a_{\rho}(p_u)^{r_u}$ remains a constant.

Suppose $p_1$ is unramified, and fix the splitting types of $p_2,\cdots,p_u$, and let $\text{Frob}_{p_1}$ runs through the conjugacy classes of $G$. Then by (\ref{estimate}), the sum of such $N$ partitions is
\begin{equation}\label{P}
\sum_C \left(\frac{|C|a_\rho(p_1)}{|G|(1+f(p_1))} A(\mathcal S_2,...,\mathcal S_{u})X + O((p_1\cdots p_u)^\kappa X^\delta) \right),
\end{equation}
for a constant $A(\mathcal S_2,...,\mathcal S_u)$.
Let $\chi_\rho$ be the character of $\rho$. Then $a_{\rho}(p)=\chi_{\rho}(g)$, where $g=\text{Frob}_p$. By orthogonality of characters,
$\sum_C |C| a_{\rho}(p_1)=\sum_{g\in G} \chi_\rho(g)=0$. Hence the above sum is
$O((p_1\cdots p_u)^\kappa X^\delta)$. The contribution from these $N$ partitions to $(\ref{composition-bound})$ is, 
\begin{eqnarray*}
&& \ll X^\delta \frac {(2r)!}{r_1!\cdots r_u!} \frac 1{u!}  {\sum}_{p_1,...,p_u}^{(2)} p_1^{\kappa-1}\cdots p_m^{\kappa-1} p_{m+1}^{\kappa-r_{m+1}}\cdots p_u^{\kappa-r_{u}} \\
&& \ll X^\delta \frac {(2r)!}{r_1!\cdots r_u!} \frac 1{u!} \prod_{i=1}^m \left(\sum_{y<p_i<x} p_i^{\kappa-1}\right) \prod_{i=m+1}^u \left(\sum_{y<p_i<x} p_i^{k-r_i}\right) \\
&& \ll  2^u X^\delta \frac {(2r)!}{r_1!\cdots r_u!}\frac{1}{u!} \frac{ x^{u \kappa }}{ (\log x)^u} \ll 2^u X^{\delta} \frac{(2r)!}{r!} \frac{ x^{u \kappa }}{ (\log x)^u} y^{m+r-u}(\log y)^{u-r}\ll 2^u X^{\delta} \frac{(2r)!}{r!} (\log X)^{ u \kappa \beta + r}. 
\end{eqnarray*}
Here we used Lemma \ref{tech lemma} for the second last inequality. 

Hence the contribution from the cases when $p_j$ is unramified for some $j \leq m$, is
\begin{eqnarray*}
 &&\ll (N+w)^{u} 2^{u} X^{\delta} \frac {(2r)!}{r!}(\log X)^{u \kappa \beta + r} 
\ll (N+w)^{2r} 2^{2r} X^{\delta} \frac {(2r)!}{r!}(\log X)^{2r(\kappa \beta + 1)}.
\end{eqnarray*}

If we choose $c_2$ sufficiently small, for example, taking $c_2=\frac{1-\delta}{20(\kappa \beta+1)}$,
$$
(N+w)^{2r} 2^{2r} X^{\delta} \frac {(2r)!}{r!}(\log X)^{2 r (\kappa \beta+1) }
\ll d^{2r} X \frac {(2r)!}{r!} \frac {2^{2r}}{(y\log y)^r}.
$$
Hence we verified $(\ref{composition-bound})$.

Now, we assume that $p_1,p_2,\cdots,p_m$ are all ramified. Then by (\ref{estimate}), the number of elements in the set of $\rho\in L(X)$ with the local condition $\mathcal S_{p_i,r}$ for $i=1,\dots,m$, is
$$
\prod_{i=1}^m \frac {f(p_i)}{1+f(p_i)} A(r_2)X + O( (p_1\cdots p_m)^\kappa X^{\delta}),
$$

Since $\frac {f(p)}{1+f(p)}\ll \frac 1p$, by the trivial bound, the main term contributes to (\ref{composition-bound})
\begin{eqnarray*}
&& X d^{2r} {\sum}_{p_1,...,p_u}^{(2)} \frac 1{p_1^{2}\cdots p_m^{2} p_{m+1}^{r_{m+1}}\cdots p_u^{r_u}} 
\ll X d^{2r} \prod_{i=1}^m \left(\sum_{y<p_i<x} p_i^{-2}\right) \prod_{i=m+1}^u \left(\sum_{y<p_i<x} p_i^{-r_i}\right) \\
&& \ll X d^{2r} 2^{2r} (y\log y)^{-r} \frac {y^u}{y^{m+r}}\cdot \frac {(\log y)^r}{(\log y)^u}.
\end{eqnarray*}
By Lemma \ref{tech lemma}, (\ref{composition-bound}) is verified.

The contribution of the error term $O((p_1\cdots p_m)^\kappa X^{\delta})$ is the same as when $p_1$ is unramified.
\end{proof}

Now take $y=c_1\log X$, and $r=c_2\frac {\log x}{\log\log X}$. Then from Proposition \ref{MP}, the number of $\rho\in L(X)$ such that 
$\left|\sum_{y<p<x} \frac {a_{\rho}(p)}p\right|>\frac 1{(\log\log X)^{1/2}}$, is
\begin{equation}\label{error}
\ll X e^{-c' \frac {\log X}{\log\log X}\log\log\log X},
\end{equation}
for some $c'>0$. This proves Proposition \ref{Lam}.

\subsection{Infinite family of number fields with extreme residues}

Let $C$ be a conjugacy class of $S_{d+1}$, and $\mathcal S=(S_{p,C})_{p\leq y}$ be the set of local conditions  
such that for every prime\ $p\leq y$, $Frob_p\in C$. 
We denote $L(X,\mathcal S)^{r_2}$ by $L(X,\mathcal S)$.
Conjecture \ref{counting} says that
\begin{eqnarray*}
|L(X, \mathcal S)|= A(r_2)X \prod_{p\leq y} \frac {\frac {|C|}{|S_{d+1}|}}{1+f(p)} +O\left(\Bigg(\prod_{p \leq y}p\Bigg)^\gamma X^\delta\right).
\end{eqnarray*}

The main term is 
\begin{equation}\label{main-term}
A(r_2) \frac X{\log y} \exp\left(-\log \frac {|S_{d+1}|}{|C|} \cdot \frac {\log X}{\log\log X}\right).
\end{equation}
This is larger than (\ref{error}). Also we may assume that almost all $L$-functions in $L(X, \mathcal S)$ have the desired zero-free region of the form in 
Proposition \ref{App2}. Hence, by Proposition \ref{Lam}, except $O(X e^{-c' \frac {\log X}{\log\log X}\log\log\log X})$ fields,
$$L(1,\rho)=\prod_{p\leq y\atop Frob_p\in C} \prod_{i=1}^d (1-\alpha_i p^{-1})^{-1} \left(1+O\left(\frac 1{(\log\log |D_K|^{\frac 12}}\right)\right).
$$

By taking $C=1$, we obtain an infinite family of number fields with the upper bound. On the other hand, by taking $C=(1,\cdots,d+1)$, we obtain an infinite family of number fields with the lower bound. This proves Theorem \ref{main}.

In a similar way, for each $0\leq i\leq d$, $d-i$ even, we can construct an infinite family of number fields with the residue 
$$\zeta(2)^{\frac {d-i}2} e^{\gamma i}(\log\log |D_K|)^i (1+o(1)).
$$
In particular we obtain an infinite family of number fields with bounded residues by taking 
$$C=\begin{cases} (1,2)(3,4)\cdots (d-1, d), &\text{if $d$ is even}\\ (1,2)(3,4)\cdots (d-4, d-3)(d-2, d-1, d), &\text{if $d$ is odd}\end{cases}.
$$
for which 
\begin{eqnarray*}
Res_{s=1}\zeta_K(s)=L(1,\rho)= \begin{cases} \zeta(2)^{\frac d2}(1+o(1)), &\text{if $d$ is even}\\ \zeta(2)^{\frac {d-3}2}\zeta(3)(1+o(1)), &\text{if $d\geq 3$ is odd}.\end{cases},
\end{eqnarray*}
and it proves Theorem \ref{bounded value}.

\end{document}